\documentclass[12pt,a4paper]{article}

\usepackage[affil-it]{authblk}
\usepackage[utf8]{inputenc}
\usepackage{latexsym,amsthm,amsmath,amssymb}
\usepackage{mathtools}
\usepackage{color}
\usepackage[colorlinks]{hyperref}
\usepackage{graphics}            
\usepackage{graphicx}
\usepackage{subcaption}
\usepackage{lineno}

\newcommand{\RR}{\mathbb{R}}
\newcommand{\RF}{\mathbb{R}_\mathcal{F}}

\newcommand{\RFA}{\mathbb{R}_{\mathcal{F}(A)}}
\newtheorem{theorem}{Theorem}
\newtheorem{definition}{Definition}
\newtheorem{proposition}{Proposition}
\newtheorem{lemma}{Lemma}

\newtheorem{corollary}{Corollary}

\begin{document}

 \title{Fuzzy Variational Calculus in Linearly Correlated Space: Part I}

\author{Gastão S. F. Frederico\thanks{gastao.frederico@ua.pt}}
\affil{Federal University of Ceará, Campus Russas, Brazil}

\author{Estevão~Esmi\thanks{eesmi@unicamp.br} \; and \; Laécio~C.~Barros\thanks{laeciocb@ime.unicamp.br}}
\affil{Department of Applied Mathematics, Universidade Estadual de Campinas, Campinas, Brazil}

\date{}

\maketitle

{\small 
\section*{Abstract} 

This article is the first part of series of articles that aim to present the foundations for fuzzy variational calculus for functions taking values in the space of linearly correlated fuzzy numbers $\RFA$. 
Recall that the space $\RFA$ is composed by all sums of a real number $r$ and a fuzzy number $qA$, where  $A$ is a given asymmetric fuzzy number. Two advantages of this space are that it can be equipped with a Banach space structure and, similar to additive models in Statistics, its elements can be interpreted as the sum of a deterministic expected/predictable value with an uncertain/noise component. 
The foundation of variational calculus theory requires the definition and establishement of many concepts and results. This aritcle presents a total order relation on $\RFA$ for wihch the notions of local minimal and maximal of a $\RFA$-valued function $f$ can be derived. We present a fuzzy version of the first and second optimality conditions in terms of derivatives of $f$. 
Finally, we present a generalized  fuzzy version of du Bois--Reymond lemma which is essential in variational calculus theory.

{\bf keyword: } Fuzzy Variational Calculus; Linearly correlated fuzzy numbers; 
Order relation; Optimality conditions; Lagrange lemma; du Bois-Reymond lemma.
}

\section{Introduction}

Nowadays, there are several approach about calculus for fuzzy number-valued functions and all them extends the classical calculus for real functions. As a consequence, this can also be seen in  fuzzy approach to calculus of variations by considering fuzzy functions and their derivatives into the variational integrals to be extremized. This may occur naturally in many problems of physics and mechanics.  Recently, Farhadinia \cite{farhadinia2011necessary} studied necessary optimality conditions for fuzzy variational problems using the fuzzy differentiability concept due to Buckley and Feuring \cite{buckley2000fuzzy}. In \cite{fard2012note}, Fard and Zadeh use the concept of $\alpha$–differentiability to obtain an extended fuzzy Euler-Lagrange condition. Subsequently, Fard {\it et al.} present the fuzzy Euler–Lagrange conditions for fuzzy constrained and unconstrained variational problems under the generalized Hukuhara ($gH$) differentiability \cite{fard2014fuzzy}. 
Base on the notion of $gH$-differentiability, Zhang {\it et al.} study Euler-Lagrange equations for fuzzy fractional variational problems under gH-Atangana-Baleanu Differentiability \cite{zhang2018generalized}. 
However, their approach does not bring light in the concept of extremals (i.e., the concepts of minimal and maximal), since there is not a consensual (total) order relation on fuzzy numbers. 

In the literature, there are many proposals of order relation on fuzzy numbers, for instance, see \cite{goetschel86,neres2024alpha,valvis2009new,zumelzu2022admissible,wang2014total} and the references therein. Despite all these order relations, one of them has been usual in optimization problems that is defined in terms of endpoints of the $\alpha$-levels \cite{hernandez2022review,osuna2019optimality,osuna2016necessary}. In \cite{van2020convergence}, authors investigate optimizations problems in terms of other order relation that is provided in \cite{goetschel86}. Recently, Antczak presents optimality conditions for fuzzy optimization problems under granular differentiable fuzzy objective functions \cite{antczak2024optimality}.  

Our study do not deal with the general class of fuzzy number ($\RF$) nor does it consider the fuzzy calculus from generalized Hukuhara derivative. Instead, here we use the subclass $\RFA$ of $\RF$ which is a subset composed by linearly correlated fuzzy numbers or linearly interactive fuzzy numbers (in fact, the precise term it should be linearly
interactive fuzzy variables, the interested reader can see \cite{de2023interactive,costa2023theory}.
Here, we present a different and new total order relation on $\RFA$ for which we state necessary and sufficient optimality conditions. % and other properties, 
%Furthermore, we also states other properties  

In optimization problems, the objective function often involves squaring operations. 
Similarly, this can also be the case when we are dealing with fuzzy optimization problems. 
To this propose, it is reasonable to expected that the square of a fuzzy number is greater or equal than zero. 
However, in general, such a property does not hold true if we define the square of a fuzzy number in terms of the usual multiplication operation. Moreover, $\RFA$ is not closed under the usual multiplication operation.  
In contrast, Longo {\it et al.} defined the $\Psi$-cross product on $\RFA$ for which $\RFA$ is closed under it \cite{longo2022cross}. 
In this work, we define a total order on $\RFA$ such that the $\Psi$-cross product of any element of $\RFA$ by itself is greater or equal zero.  
These results are important to establish the fuzzy version of the Langrange du Bois-Reymond lemmas that are  two fundamental results that have major implications in several areas of mathematics including variational calculus. 
To the best of our knowledge, this is the first time that these lemmas are established in the fuzzy context. 

%Longo {\it et al.} defined the $\Psi$-cross product between two elements of $\RFA$ whenere $1$-level of $A$ is an unit set \cite{longo2022cross}.

The paper is organized as follows. Section 2 presents some basic notions of fuzzy set theory and properties of the subclass of fuzzy number $\RFA$. 
A brief recall of the calculus theory for $\RFA$-valued functions is provided in this section.
In Section 3, we present an order relation on $\RFA$ and states some of its properties with respect to the  $\Psi$-cross product. 
Sufficient and necessary optimality conditions regarding the proposed order relation is provided in  Section 4. In Section 5, we  present the main results of this first article which includes a fuzzy version of the well-known  Lagrange and du Bois-Reymond lemmas that will serve as a basis for establishing a theory of fuzzy variational calculus in subsequent articles of this series of articles.  
Finally, we conclude this paper present some final comments and concluding remarks in the last section.  

\section{Mathematical Background on Fuzzy Sets and Calculus on  $\RFA$}\label{basic}

A fuzzy (sub)set $A$ of the universe $X\neq\emptyset$ is described by its membership function 
$\mu_{_A}:X\longrightarrow [0,1]$, where $\mu_{A}(x)$ means the degree to which $x$ belongs to $A$. For notational convenience, we write $A(x)$ instead of $\mu_A(x)$. % \cite{zadeh65}. 
%We use the symbol $\mathbb{F}(X)$ to denote the set of all fuzzy subsets of $X$.
The $\alpha$-levels of the fuzzy subset $A$ are classical subsets defined as follows: 
%\begin{center}
	$[A]_{\alpha}=\{x\in X:
A(x)\geq \alpha \}$ for  $0<\alpha\leq 1$ and, when $X$ is a topological space, 
$[A]_0=\overline{\{x\in X:A(x)>0\}}$, where $\overline{Y}$ stands for the closure of $Y \subseteq X$.
%\end{center}
Here, we focus on a particular classes of fuzzy sets of $\RR$ called fuzzy numbers, which will be denoted by $\RF$.

\begin{definition}[Fuzzy number] \cite{de2017first,esmi2021solutions} %\cite{Vinicius} 
The fuzzy subset $A$ of $\mathbb{R}$ is a  fuzzy number if 
	\begin{itemize}
		\item [i.] all their $\alpha$-levels are closed and nonempty intervals of $\mathbb{R}$, and
		\item[ii.] the support of $A$, $supp(A)=\{x\in \mathbb{R}:A(x)>0\}$, is bounded.
	\end{itemize}
\end{definition}

%Let $\RF$ be the set of fuzzy numbers. 
For every %$x\in \RR$ and 
$\alpha \in [0,1]$, we denote the $\alpha$-level of $A \in \RF$  by $[A]_{\alpha} := [\underline{a}_\alpha, \overline{a}_\alpha]$. 
Note that every real number $r$ can be regarded as a fuzzy number such that its $\alpha$-level is given by $[r,r]$ for all $\alpha \in [0,1].$ Therefore, the set of real numbers can be embedded in $\RF$ and we simply denote it by the symbol $\RR \subset \RF$. 

%The class of fuzzy numbers is a complete metric space with the following metric  \cite{diamond2000metric}:
%$$d_{\infty }(A,B)=\sup_{0\leq \alpha \leq 1}\max \{|a^{-}_{\alpha}-b^{-}_{\alpha}|,|a^{+}_{\alpha}-b^{+}_{\alpha}|\}$$
%where $A,B \in \RF$. 

%Some examples of fuzzy numbers that we use in our simulations are the classes of trapezoidal and triangular fuzzy numbers. A fuzzy number $A$ is said to trapezoidal if there exist real numbers $a\leq b\leq c \leq d$ such that $[A]_\alpha =[a+\alpha(b-a), d-\alpha(d-c)]$ for all $\alpha \in [0,1]$. In this case, $A$ is also denoted by $(a;b;c;d)$. 
%We speak of a triangular fuzzy number if $b=c$ and denote $A$ by $(a;b;d)$ instead.

For every $C,D \in \RF$ and $\lambda \in \RR$, the usual sum and product of $C$ and $D$ and the product of $\lambda$ and $C$ are  defined level-wise respecatively by:
\begin{equation}\label{somausual}
[C + D]_\alpha = [\underline{c}_\alpha + \underline{d}_\alpha, \overline{c}_\alpha + \overline{d}_\alpha],
\end{equation}
\begin{equation}\label{produtousual}
[C \cdot D]_\alpha = [\min P_\alpha, \max P_\alpha],
\end{equation}
and 
\begin{equation}\label{prodescalarusual}
[\lambda C]_\alpha = \left\lbrace \begin{array}{cl}
 \left[\lambda \underline{c}_\alpha, \lambda\overline{c}_\alpha\right]    &, \mbox{ if }  \lambda \geq 0, \\
 \left[\lambda \overline{c}_\alpha, \lambda\underline{c}_\alpha\right]    &, \mbox{ if }  \lambda < 0. 
\end{array}\right.     
\end{equation}
 where $P_\alpha = \{\underline{c}_\alpha\underline{d}_\alpha, \underline{c}_\alpha\overline{d}_\alpha,
\overline{c}_\alpha\underline{d}_\alpha,\overline{c}_\alpha\overline{d}_\alpha\}.$ 

The class of fuzzy numbers with these operations does not form a vector space because every 
$B \in \RF \setminus \RR$ does not have an opposite (i.e. an additive inverse) and the distributive 
equality $(\lambda + \mu)B = \lambda B + \mu B$, $\lambda, \mu \in \RR$, holds true only if $\mu\lambda \geq 0$. Below, we recall the subclass $\RFA$ of $\RF$ for which, under weak conditions, forms a vector space with a different notion of addition.  

\subsection{The Space $\RFA$}

%for all $B,C \in \RF$ and $\lambda \in \R$.
For every fuzzy number $A$, we can consider the space $\RFA$ defined as follows:
\begin{equation}
\RFA = \{r + qA \mid r,q \in \RR \}.
\end{equation}
In \cite{esmi2018frechet}, authors have proved that the mapping $\Psi$ given by $(r,q)\mapsto \Psi(r,q) = r + qA$ is injective if, and only if, $A$ satisfies the property of asymmetry.
A fuzzy number $A$ is said to be asymmetric (or non-symmetric) if for all $x \in \RR$ there exists $y \in \RR$ such that $A(x-y) \neq A(x+y)$. The following theorem states that  $\RFA$ forms a vector space with the addition operation induced by the mapping $\Psi$. 

\begin{theorem} \cite{esmi2018frechet} \label{thm:rfa_space}
Let $A$ be an asymmetric fuzzy number. The set $\RFA$ is a real vector space with the usual scalar product and the addition operation given as follows. 
For every $B=(r_B + q_BA) \in \RFA$ and $C=(r_C + q_CA) \in \RFA$, the sum of $B$ and $C$ 
is an element $B \oplus C$ of $\RFA$ defined by 
\begin{equation}
B \oplus C  = \Psi\left(\Psi^{-1}(B) +  \Psi^{-1}(C) \right) = (r_B+r_C) + (q_B+q_C)A.     
\end{equation}

Moreover, $\RFA$ is isomorphic to $\RR^2$ with the isomorphism $(r,q)\mapsto r + qA$. 
\end{theorem}

Theorem \ref{thm:rfa_space} ensures that $\RFA$ is a 2-dimension vector space.  
Furthermore, it is worth noting that one can easily verify that for $\lambda \in \RR$ and $B \in \RFA$ we have $\lambda B =  \Psi\left(\lambda \Psi^{-1}(B)\right).$ 
In other words, the usual scalar product restrict to $\RFA$ coincides with the scalar product induced by $\Psi$.

For the propose of this work, we only focus on the case the space $\RFA$ forms a real vector space.
Therefore, from now on, unless otherwise stated, we assume that the fuzzy number $A$ that generates the space $\RFA$ is asymmetric. It is worth noting that, from a practical point of view, the assumption of asymmetry is not a strong condition but indeed a weaker one since the set of asymmetric fuzzy numbers is open and dense in $\RF$ w.r.t. the Hausdorff metric \cite{esmi2021banach}.
In addition, as usual in vector spaces, for every $C, B \in \RFA$, the additive inverse (or the opposite) of $B$ is given by $-B = (-1)B$ and the difference of $C$ and $B$ is defined 
by $C\ominus B = C \oplus (-B)$.

The space $\RFA$ is not closed under the usual multiplication operation on fuzzy multiplication. 
In order to equip $\RFA$ with a multiplication operation,  
Longo {\it et al.} defined the $\Psi$-cross product between two elements of $\RFA$ whenever the $1$-level of $A$ is an unit set \cite{longo2022cross}.
%The set of all fuzzy numbers such that its $1$-level is an unit set is denoted by $\RFH$. If $A \in \RFH$, then we use the symbol $\RFAH$ instead of $\RFA$.  In this case we have

\begin{definition}\label{def:multiplication} \cite{longo2022cross}
For every $A \in \RF$, with $[A]_1 = \{a\}$, and $B, C \in \RFA$ with $[B]_1 = \{b\}$ and $[C]_1 = \{c\}$. The $\Psi$-cross product of $B$ and $C$ is defined by
\begin{equation}\label{eq:cross_product_original}
B \odot C = cB + bC - bc. 
\end{equation}
\end{definition}

The $\Psi$-cross product of two elements of $B,C \in \RFA$ with $[B]_1 = \{b\}$ and 
$[C]_1 = \{c\}$ can be seen as extension to $\RFA \times \RFA$ of the linear approximation of the product function around the pair $(b,c)$. More precise, let $f(x,y) = xy$ for all $x,y \in \RR$, the linear approximation of $f$ at $(b,c)$ is the first-order Taylor polynomial of $f$ around $(b,c)$ given by the function $g(x,y) = bc + c(x-b) + b(y-c) = cx + by - bc$. 
If $[A]_1 = \{a_m\}$, $B=r_B + q_BA \in \RFA$, and $C=r_C + q_CA \in \RFA$, then 
the $\Psi$-cross product of $B$ and $C$ can be rewritten as follows \cite{longo2022cross}:
\begin{equation}\label{eq:cross_product}
B \odot C = (r_Br_C - a_m^2q_Bq_C) + (r_Bq_C + r_Cq_B + 2a_mq_Bq_C)A. 
\end{equation}
%\end{definition}

From now on, unless stated otherwise, we assume that the asymmetric fuzzy number $A$ that generates the space $\RFA$ is such that $[A]_1 =\{a_m\}$. In other words, we assume that the $\Psi$-cross product is well-defined on $\RFA$. 

We finish this section, pointing out that, from a modelling point of view, a fuzzy number $A$ such that $a_m = 0$ can be interpreted a fuzzy noise centered at zero. In this case, one can understand an element $B = (r +qA) \in \RFA$ as fuzzy number ``around $r$'' since $B$ is the sum the deterministic value $r$ and the fuzzy noise $qA$, where $q$ is a factor that modulates the length of the (fuzzy) noise around $r$. This kind of interpretation has a clear parallel with the additive models found in statistics.
%Thus, in what follows we assume that $a_m = 0$, that is, a fuzzy number centered at zero.  
The assumption that $A$ is centered at zero  
%In fact, this 
is not a strong assumption, since for every $\tilde{A} \in \RF$ the fuzzy number 
$A = -\tilde{a}_m + \tilde{A}$, with $\tilde{a}_m = 0.5(\underline{\tilde{a}}_1 + \overline{\tilde{a}}_1)$, is centered at zero and belongs to $\mathbb{R}_{\mathcal{F(\tilde{A})}}$ with $\RFA = \mathbb{R}_{\mathcal{F(\tilde{A})}}$ \cite{esmi2018frechet}.

\subsection{Calculus on $\RFA$-valued Functions}

%The space $\RFA$ can be endowed with a norm \cite{esmi2018frechet}
%In \cite{esmi2018frechet}, authors propose a norm on $\RFA$ induced by the 1-norm of $\RR^2$. 
%Instead of using this norm, we shall 
In what follows, we shall equip the space $\RFA$ with a notion of norm that can be related with the notion of order propose in the next section (see Lemma \ref{lem:first_condition}). The following theorem provides another notion of norm on $\RFA$ from that  presented in \cite{esmi2018frechet}. 

\begin{theorem}\label{thm:norm}
Let $A$ be an asymmetric fuzzy number and $a_m = 0.5(\underline{a}_1 + \overline{a}_1)$. The space $\RFA$ is a normed space with the norm $\| \cdot \|: \RFA \to \RR$ given for every $B = (r + qA) \in \RFA$ as follows:
\begin{equation}\label{eq:norm}
\|B\| = |q| + |r + qa_m|.    
\end{equation}
\end{theorem}
\begin{proof}
First of all, Equation \eqref{eq:norm} is well-defined, since for every $B \in \RFA$ there exist unique $r$ and $q$ such that $B = r + qA$. 

By definition, we have that $\|B\| \geq 0$ for all $B \in \RFA$. Moreover, if $\|B\| = 0$, then  
$|q|=0$ and $|r + qa_m| = 0$. The first equality ($|q|=0$) implies that $q=0$. This also implies that $r=0$ because the second equality becomes $|r| = 0$. Therefore, $\|B\| = 0 \Leftrightarrow B = 0$. 

For $\lambda \in \RR$, we have 
\begin{eqnarray*}
\| \lambda B \| & = & \| (\lambda r) + (\lambda q) A \| \\
 & = & |\lambda q| +  |\lambda r +  \lambda q a_m | \\
 & = & |\lambda|(|q| +  |r +  q a_m | = |\lambda|\|B\|.
\end{eqnarray*}

Let $C = s + pA$. It follows that 
\begin{eqnarray*}
\| B \oplus C \| & = & \| (r + s) + (q + p) A \| \\
 & = & |q + p| +  |(r + s) + (q + p) a_m | \\
 & \leq & |q| + |p| + |r + qa_m | + |s + pa_m| = \|B\| + \|C\|. 
\end{eqnarray*}
\end{proof}

Since the concept of norm induces a metric on the 2-dimension Banach space $\RFA$, the concepts and notations of convergence, limit, and continuity on $\RFA$ are defined as usual for any Banach space. 
%The following lemma states that the inequality is preserved in the limit.

% because  is indeed a 2-dimmension Banach space. 
%From the isomorphism between $\RR^2$ and $\RFA$, the space  
Below, we recall the definition of derivative for $\RFA$-valued functions. 

\begin{definition}\label{def:derivative}
Given $f:I\to \RFA$ with $I=(a,b) \subset \RR$. We say that $f$ is differentiable at $t \in I$ if there exists $f'(t) \in \RFA$ such that 
\begin{equation}
 \lim_{h\to 0} \frac{f(t+h) \ominus f(t)}{h} = f'(t).   
\end{equation}
In this case, $f'(t)$ is said to be the derivative of $f$ at $t$. If $f$ is differentiable at all points of $I$, then we say that $f$ is differentiable in $I$. 
\end{definition}

A well-known fact is that every norm in a finite dimension Banach space are equivalent. 
Thus, the norm defined in Equation \eqref{eq:norm} is equivalent to the norm provided in \cite{esmi2018frechet}. It means that the convergence of a sequence or a function at some point w.r.t. one of these two norm implies in the converge w.r.t. the other norm. Moreover, it converges to the same limit in both norms. As a consequence all results obtained with any other norm on $\RFA$ also hold for the proposed norm.  For example, from this equivalence and from Corollary 5 in \cite{esmi2020calculus}, we obtain the following proposition. 

\begin{proposition}\label{prop:derivative}
Let $f:I\to \RFA$ be such that $f(t) = r(t) + q(t)A$ for all $t\in I$, where $r,q:I\to \RR$ and $I=(a,b) \subset \RR$. The function $f$ is differentiable at $t \in I$ if, and only if, $r$ and $q$ are differentiable at $t$. Moreover, we have that 
\begin{equation}
f'(t) = r'(t) + q'(t)A.     
\end{equation}
\end{proposition}

The notions and notations of partial derivates and derivative of higher orders for functions taking values in $\RFA$ can be naturally defined as in the classical case. 
%In fact, a function $f(t) = x(t) + y(t)A$ has     
Since $\RFA$ is a 2-dimmension Banach space, we can also consider the well-establised notion of Riemman integral for functions from an interval of $\RR$ to a Banach space \cite{gordon79}. This integral regards as an operator on the space of integral functions satisfies many interesting properties such as linearity, continuidade, etc \cite{gordon79}. 
In the particular context of the functions from $[a,b]\subset \RR$ to $\RFA$, 
the precise definiton of Riemman integral can be found in 
 \cite{santo2020calculus,simoes2023interactive}. 
 An interesting and useful result is that the existence of this integral is caracterized by the existence of two integrals of real-valued funcitons. The following proposition clarifies this fact. 

\begin{proposition}\cite{santo2020calculus} \label{prop:integral}
Let $f:[a,b]\to \RFA$ be such that $f(t) = r(t) + q(t)A$ for all $t\in [a,b]$, where $r,q:[a,b]\to \RR$. The functin $f$ is integrable on $[a,b]$ if, and only if, $r$ and $q$ are integrable on $[a,b]$. Moreover, we have that 
\begin{equation}
\int_a^b f(t)dt = \left(\int_a^b r(t)dt\right) + \left(\int_a^b q(t)dt\right)A.     
\end{equation}
\end{proposition}

Similarly to the crisp case, the symbols $\textnormal{C}\left([a,b],\RFA\right)$, 
$ \textnormal{C}{}^k\left([a,b],\RFA\right)$, 
$\textnormal{AC}\left([a,b],\RFA\right)$, and
$\textnormal{Lip}\left([a,b],\RFA\right)$
stands to the classes of functions from $[a,b]$ to $\RFA$ that are 
continuous, $k$ times continuously  differentiable, absolute continuous, and Lipschitz, respectively.  

Based on Propositions \ref{prop:derivative} and \ref{prop:integral}, one can easily prove the 
linearity properties for the derivative and integral operators and the fuzzy version of the fundamental calculus theorems. Here, as in classical case, the arithmetic operations between $\RFA$-valued functions are defined pointwise.  

\begin{theorem} \cite{santo2020calculus,shen2022calculus}
Let $A$ be an asymmetric fuzzy number, $f,g:[a,b]\to \RFA$ and $\lambda \in \RR$. 
The following statements are valid:
\begin{itemize}
\item[a)] If $f,g$ are differentiable in $(a,b)$, then 
$$
(f \oplus \lambda g)'(t) = f'(t) \oplus \lambda g'(t), \qquad \forall t \in (a,b).
$$
\item[b)] If $f,g$ are integrable in $[a,b]$, then 
$$
\int_a^b f(t) \oplus \lambda g(t)dt = \int_a^b f(t) dt \oplus \lambda \int_a^b g(t)dt.
$$
\item[c)] If $f$ is integrable in $[a,b]$, then the function $F(t) = \int_a^t f(s)ds$ is continuous 
and $F'(t) = f(t)$ almost everywhere in $[a,b]$. 
\item[d)] If $f$ is differentiable in $(a,b)$ and $f'$ is integrable in $[a,b]$, then 
$$
\int_a^b f'(t) dt = f(b) \ominus f(a) = f(y)\Big|_a^b .
$$
\end{itemize}
\end{theorem}

The following theorem presents the Product rule (or Leibniz rule) and 
Integration by parts with respect to $\Psi$-cross product.  

\begin{theorem}\cite{simoes2023interactive}
Let $A$ be an asymmetric fuzzy number and $f,g:[a,b]\to \RFA$. 
The following statements are valid:
\begin{itemize}
\item {\bf Product Rule:} If $f,g$ are differentiable in $(a,b)$, then 
$$
(f\odot g)'(t) = f(t) \odot g'(t) \oplus f'(t) \odot g(t), \quad \forall t \in (a,b).
$$
\item {\bf Integration by parts:} If $f,g$ are differentiable in $(a,b)$ with $f,g,f',g'$ integrable in $[a,b]$, then 
$$
\int_a^b f(t)\odot g'(t)dt = f(t) \odot g(t) \Big|_a^b - \int_a^b f'(t)\odot g(t)dt.
$$
\end{itemize}
\end{theorem}

\section{Order Relation on $\RFA$}

Minimal and maximal elements are basic concepts in optmization/variational problems that arise from a given order relation. The concept of order relation on a nonempty set $X$ is nothing more than a (binary) relation on $X$, that is, a subset of $X\times X$, which satisfies the conditions of the 
definition below. In addition, an order relation is usually denoted by $\leq$ and the symbol $x \leq y$ means that $(x,y)$ belongs to the relation $\leq$ and it is read as ``$x$ is less or equal to $y$''. If $(x,y) \not\in \leq$, we say that ``$x$ is not less or equal to $y$'' and this fact is denoted by $x\not\leq y$.

\begin{definition}\cite{birkhoff93}\label{def:order}
An order relation on a nonempty set $X$ is a relation on $X$ that satisfies the following three properties for all $x,y,z \in X$:
\begin{itemize}
    \item[i)] $x \leq x$;
    \item[ii)] if $x \leq y$ and $y \leq x$, then $x=y$;
    \item[iii)] $x \leq y$ and $y \leq z$, then $x\leq z$.
\end{itemize}

If for every $x,y \in X$ we have $x\leq y$ or $y \leq x$, then $X$ is said to be totally ordered or a chain, otherwise, $X$ is said to be partially ordered or, for short, a poset. 
\end{definition}

Well-known examples of chains are the set of real numbers with the usual order and the power set of a given set with the inclusion order. The following proposition provides a order relation on $\RFA$ for which $\RFA$ becomes a totally ordered set.

\begin{proposition}\label{Proorder}
Let $A$ be an asymmetric fuzzy number and $a_m = 0.5(\underline{a}_1 + \overline{a}_1)$. The space $\RFA$ is a totally ordered set with the order relation given for every $B = (r_B + q_BA) \in \RFA$ and $C = (r_C + q_CA) \in \RFA$ by  
\begin{equation}\label{eq:order}
B \leq C \Leftrightarrow \left\lbrace \begin{array}{cc}
r_B + q_Ba_m < r_C + q_Ca_m, & \mbox(I) \\
     \mbox{ or }  & \\
r_B + q_Ba_m = r_C + q_Ca_m  \mbox{ and }  |q_B| < |q_C|, & \mbox(II) \\
     \mbox{ or } & \\
r_B + q_Ba_m = r_C + q_Ca_m  \mbox{ and }  |q_B| = |q_C|   \mbox{ and }  q_B \leq q_C & \mbox(III).
\end{array} \right.   
\end{equation}
\end{proposition}
\begin{proof}
Let $B,C,D \in \RFA$ be such that $B = r_B + q_BA$, $C = r_C + q_CA$, and $C = r_C + q_CA$. 
Note that $B \leq B$ since the third condition of right side of \eqref{eq:order} is satisfied. 

Now, let us analyse Item ii) of Definition \ref{def:order}. 
Suppose that  $B\leq C$. Note that if  the first condition of the right side of \eqref{eq:order} is satisfied, then, we have that $C \not\leq B$. Hence, this case is not possible if the hypothesis of Item ii) of Definition \ref{def:order} is satisfied. Now, if the second condition of the right side of \eqref{eq:order} is satisfied, then, $C \not\leq B$. Again, in this case, the hypothesis of Item ii) can not be satisfied. 
Thus, the only possibility is that the third condition is satisfied if the hypothesis of Item ii) holds true. 
Using a similar reasoning for $C \leq B$, we conclude that 
the hypothesis that $B\leq C$ and $C \leq B$ occurs if, and only if, we have 
\[
\left\lbrace\begin{array}{c}
r_B + q_Ba_m = r_C + q_Ca_m  \mbox{ and }  |q_B| = |q_C|   \mbox{ and }  q_B \leq q_C  \\
r_B + q_Ba_m = r_C + q_Ca_m  \mbox{ and }  |q_B| = |q_C|   \mbox{ and }  q_C \leq q_B. 
\end{array}\right.
\]
This implies that $q_B = q_C$ and $r_B = r_C$. Therefore, we obtain that $B = C$.

Suppose that $B \leq C$ and $C \leq D$. 
Consider the case where $r_B + q_Ba_m < r_C + q_Ca_m$. Since $C \leq D$, we have that 
$r_C + q_Ca_m \leq r_D + q_Da_m$. Thus, we $B\leq D$ because $r_B + q_Ba_m < r_D + q_Da_m$.
Similarly, if $r_C + q_Ca_m < r_D + q_Da_m$, we obtain that $r_B + q_Ba_m < r_D + q_Da_m$ and, consequently, 
$B\leq D$. 
The remain case is the case where  
$r_B + q_Ba_m = r_C + q_Ca_m = r_D + q_Da_m$. From the second and third conditions in 
the right side of \eqref{eq:order}, we have that $|q_B| \leq |q_C| \leq |q_D|$.
Thus, if $|q_B| < |q_C|$ or $|q_C| < |q_D|$, from definition, we obtain that $B\leq D$. 
Finally, if $|q_B| = |q_C| = |q_D|$, then $q_B \leq q_C  \leq q_D$, which also implies that 
$B\leq D$. 

To prove that $\RFA$ is totally order with respect to this order relation, we have to show that one of the inequalities $B\leq C$ or $C \leq B$ holds. 
If $r_B + q_Ba_m < r_C + q_Ca_m$ (or $r_C + q_Ca_m < r_B + q_Ba_m$), then $B\leq C$ (or $C \leq B$).  
If $r_B + q_Ba_m = r_C + q_Ca_m$ and $|q_B| < |q_C|$ (or $|q_C| < |q_B|$ ), then $B\leq C$ (or $C \leq B$). 
For $r_B + q_Ba_m = r_C + q_Ca_m$ and $|q_B| = |q_C|$, we have that $B\leq C$ if $q_B \leq q_C$, or  
$C \leq B$ if $q_C \leq q_B$.
\end{proof}

The following corollary states that 
%the order relation above extends the usual order relation on $\RR$ and states that 
the inequality between a real number and an a fuzzy number in $\RFA$ can  be evaluated by a simple inequality. 

\begin{corollary}\label{cor:rel_order}
Under the conditions of Proposition \ref{Proorder}, for all $B = (r + qA) \in \RFA$  and $\lambda \in \RR$, we have
\begin{itemize}
    \item[(a)] $\lambda \leq B \Leftrightarrow \lambda \leq r + qa_m$;
    \item[(b)] $B \leq \lambda  \Leftrightarrow \lambda > r + qa_m$ or $B=\lambda$.    
\end{itemize}
\end{corollary}
\begin{proof}
First of all, note that $\lambda + 0A \in \RFA$. 
If $\lambda \leq B$, then one of the three conditions in \eqref{eq:order} is satisfied and all them imply that $\lambda \leq r + qa_m.$ 
One the other hand, suppose that $\lambda \leq r + qa_m.$ If the inequality is restrict, then 
the first condition of \eqref{eq:order} holds which yields that $\lambda \leq B$. Consider the case where $\lambda = r + qa_m.$ If $q\neq 0$, then $\lambda \leq B$ because the second condition is satisfied. Now, if $q = 0$, then the third condition holds since $\lambda + 0A$. Therefore, all cases imply that  $\lambda \leq B$. 

Consider $\lambda > r + qa_m$. In this case, we obtain $B \leq \lambda$ since the condition (I) of \eqref{eq:order} is satisfied. If $B=\lambda$, then, from Proposition \ref{Proorder}, we have that $B = \lambda \leq \lambda$. On the other hand, suppose that $B \leq \lambda$. If the condition (I) of \eqref{eq:order} is satisfied, then $\lambda > r + qa_m$. If $\lambda  =  r + qa_m$, we can not have $q\neq 0$ because this would imply that $\lambda \leq B$ with $\lambda \neq B$ which combined with the hypothesis $B \leq \lambda$ produces the following contradiction: $B = \lambda \neq B$. 
Thus, in this case, we must have $q=0$ and $\lambda  =  r + qa_m = r$ which implies that $B=\lambda.$
%The proof of Item (b) follows similarly. 
\end{proof}

An immediate consequence of Corollary \ref{cor:rel_order} is that the proposed order relation on $\RFA$ extends the order relation on $\RR$. In fact, to this end is enough to consider the case where $q=0$ in Corollary \ref{cor:rel_order}. 
%Moreover, the sets of positive numbers and negative numbers in $\RFA$ can be naturually 
%For $B = (r + qA) \in \RFA$, a case of interest is to assess the inequality $0 \leq B$

In optimization problems, the objective function often involves squaring operations. 
Similarly, this can also be the case when we are dealing with fuzzy optimization problems. 
To this propose, it is reasonable to expected that the square of a fuzzy number is greater or equal than zero. 
However, in general, such a property does not hold true if we define the square of a fuzzy number in terms of the usual multiplication operation. 
Based on the notion of $\Psi$-cross product, one can define the square of an element $B \in \RFA$ as the fuzzy number $B^2 = B\odot B.$
The following proposition states $0 \leq B^2$ for all $B \in \RFA$.

\begin{proposition}\label{prop:inequality}
Let $A \in \RF$ be an asymmetric fuzzy number with $[A]_1 = \{a_m\}$.
For every $B = (r + qA) \in \RFA$, we have $0 \leq B^2$. 
Moreover, $B^2 = 0$ if, and only if, $r + qa_m = 0$
\end{proposition}
\begin{proof}
From Definition \ref{def:multiplication}, for $B=r + qA \in \RFA$ we have
\begin{equation}\label{eq:B2}
B^2 = B \odot B = (r^2 - a_m^2q^2) + (2rq + 2a_mq^2)A.    
\end{equation}

From Corollary \ref{cor:rel_order}, we have that $0 \leq B^2$ if, and only if, 
$$
0 \leq (r^2 - a_m^2q^2) + (2rq + 2a_mq^2)a_m = (r + a_mq)^2.
$$
Since $(r + a_mq)^2 \geq 0$ for every $r,q\in \RR$, we conclude that $0 \leq B^2.$
Moreover, if $B^2 = 0$, then, from equation above, we obtain $r + a_mq = 0$. Now, if $r + a_mq = 0$, then $r = -a_mq$. 
This implies $B^2 = 0$ because $(r^2 - a_m^2q^2) = 0$ and $(2rq + 2a_mq^2)= 2q(r + a_mq) = 0$. 

%If $r + a_mq \neq 0$, then the first condition of Definition \ref{def:order} holds true:
%$$(r^2 - a_m^2q^2) + (2rq + 2a_mq^2)a = (r + a_mq)^2 > 0.$$ 
%Thus, from Definition \ref{def:order}, we obtain that $0 \leq B^2$.
%
%Now, suppose that $r + qa_m = 0$. If $q \neq 0$, then $0 < |q|$ and, consequently, the condition (II) of Definition \ref{def:order} is satisfied which implies that $0 \leq B^2$. 
%If $q = 0$, then $r = 0$ and the condition (III) of Definition \ref{def:order} holds true. Again, 
%we obtain that $0 \leq B^2$. 
%
%On the one hand, if $r + qa_m = 0$ then $r^2 - a_m^2q^2 = 0$ and $(2rq + 2a_mq^2) = 2q(r + aq) = 0$. From Equation \eqref{eq:B2}, we obtain that $B^2 = 0 + 0A = 0$. 
%On the other hand, if $B^2 = 0$, then $r^2 - a_m^2q^2 = 0$ and $(2rq + 2a_mq^2) = 2q(r + aq) = 0$. 
%From the last equality, we have $q=0$ or $r + a_mq = 0$. If $q=0$, then $r^2 - a_m^2q^2 = 0$ implies that $r^2 = 0$ and, consequently, $r=0$. In this case, we we have that $r + a_mq = 0 + a_m0  0.$ 
\end{proof}

In view of Corollary \ref{cor:rel_order} and Proposition \ref{prop:inequality}, we define the following subsets of $\RFA$:
\begin{eqnarray}
\RFA^0 & = & \{B \in \RFA \mid B^2 = 0\} \\     
 & = & \{B \in \RFA \mid r + qa_m = 0  \mbox{  with  }  B = r+qA  \}, \nonumber
\end{eqnarray}
\begin{eqnarray}
\RFA^+ & = & \{B \in \RFA \mid B^2 \neq 0 \mbox{  and  } 0 < B\} \\     
 & = & \{B \in \RFA \mid r + qa_m > 0  \mbox{  with  }  B = r+qA  \}, \nonumber
\end{eqnarray}
\begin{eqnarray}
\RFA^- & = & \{B \in \RFA \mid B^2 \neq 0 \mbox{  and  } B < 0\} \\     
 & = & \{B \in \RFA \mid r + qa_m < 0  \mbox{  with  }  B = r+qA  \}. \nonumber
\end{eqnarray}
Note that the subsets $\RFA^0$, $\RFA^+$, and $\RFA^-$ are non-empty and disjoint one each other such that $\RFA = \RFA^- \cup \RFA^0 \cup \RFA^+$. The following proposition characterizes the result of  $\Psi$-cross product of two elements taking into account these subsets. 

\begin{proposition}\label{prop:sign}
Let $A \in \RF$ be an asymmetric fuzzy number with $[A]_1 = \{a_m\}$ and $B,C \in \RFA$.
The following properties hold true:
\begin{itemize}
\item[1)] If $B \in \RFA^0$, then $B\odot C \in \RFA^0$;
\item[2)] If $B,C \in \RFA^+$ or $B,C \in \RFA^-$, then $B\odot C \in \RFA^+$;
\item[3)] If $B \in \RFA^+$ and $C \in \RFA^-$, then $B\odot C \in \RFA^-$.
\end{itemize}    
\end{proposition}
\begin{proof}
First of all, we have 
$$
B \odot C = (r_Br_C - a_m^2q_Bq_C) + (r_Bq_C + r_Cq_B + 2a_mq_Bq_C)A. 
$$
where $B=r_B + q_BA$ and $C=r_C + q_CA$. 

Consider the value $\gamma$ given by 
\begin{eqnarray*}
\gamma & = ¨& (r_Br_C - a_m^2q_Bq_C) + (r_Bq_C + r_Cq_B + 2a_mq_Bq_C)a_m \\
 & = ¨& r_Br_C + r_Bq_Ca_m + r_Cq_Ba_m + a_m^2q_Bq_C \\
 & = ¨& (r_B + q_Ba_m)(r_C + q_Ca_m).
\end{eqnarray*}

\begin{itemize}
\item[1)] If $B \in \RFA^0$, then $r_B + q_Ba_m = 0$ which implies that $\gamma = 0$ and, consequently, $B \odot C \in \RFA^0$. 
\item[2)] If $B,C \in \RFA^+$ or $B,C \in \RFA^-$, then $r_B + q_Ba_m, r_C + q_Ca_m > 0$ or 
$r_B + q_Ba_m, r_C + q_Ca_m < 0$. In both cases, we have 
$\gamma = (r_B + q_Ba_m)(r_C + q_Ca_m) > 0$ which implies that $B \odot C \in \RFA^+$. 
\item[3)] If $B \in \RFA^+$ and $C \in \RFA^-$, then $r_B + q_Ba_m > 0$ and 
$r_C + q_Ca_m < 0$. This implies that $\gamma < 0$. Therefore, $B\odot C \in \RFA^-$.
\end{itemize}    
\end{proof}

The following corollary is a consequence of Proposition \ref{prop:sign} an states that 
$B^2$, or, more generally, $B^k$ with $k\geq 2$, does not belong to $\RFA^0$ if $B \not\in \RFA^0$. 

\begin{corollary}
Let $B \in \RFA$. We have that $B^2 \in \RFA^0$ if, and only if, $B^2 = 0$.   
\end{corollary}
\begin{proof}
On the one hand, if  $B^2 = 0$, then, by definition, $B^2 \in \RFA^0$.  On the other hand, suppose that $B^2 \in \RFA^0$. If $B \not\in \RFA^0$, then $B \in \RFA^-$ or $B \in \RFA^+$. From Proposition \ref{prop:sign}, we have that $B^2 \in \RFA^+$ which contradicts the hypothesis that 
$B^2 \in \RFA^0$. Therefore, we must have that $B \in \RFA^0$ which implies that $B^2 = 0$.
\end{proof}

Another interesting properties of the $\Psi$-cross product which are also consequence of Proposition \ref{prop:sign} is stated in the following corollary. 
%In particular, its Item b) provides a characterization of zero in terms of the $\Psi$-cross product operations. 

\begin{corollary}
Let $B \in \RFA$. 
\begin{itemize}
\item[a)] If there exists $C \not\in \RFA^0$ such that $B\odot C \in \RFA^0$, then $B \in \RFA^0$;
\item[b)] $B\odot C = 0$ for all $C \in \RFA \setminus \RFA^0$ if, and only if, $B = 0$.    
\end{itemize}
\end{corollary}
\begin{proof}
Item a) follows immediately from Proposition \ref{prop:sign}. Below, we prove Item b).

On the one hand, if $B=0$, then, by definition of $\Psi$-cross product, we have that 
$B\odot C = 0$ for all $C \in \RFA$. On the other hand, suppose that $B\odot C = 0$ for all $C \in \RFA \setminus \RFA^0$. Consider that $B = (r_B + q_BA)\neq 0$. In this case, we have that $r_B \neq 0$ or $q_B\neq 0$. If $r_B \neq 0$, then $C = r_B \in \RR \subset \RFA \setminus \RFA^0$ with $C^2 = r_B^2 > 0$. For this choice of $C$, we obtain that 
$$
B \odot C = r_B^2 + r_Bq_BA \neq 0.
$$
This yields a contradiction with the hypothesis that $B\odot C = 0$. Therefore, we must have that $r_B = 0$. In this case, Equation \eqref{eq:cross_product} becomes  
$$
B \odot C = -a_m^2q_Bq_C + (r_Cq_B + 2a_mq_Bq_C)A.
$$
with $C = r_C + q_CA$ and $[A]_1 = \{a_m\}$.
Now, consider the case where $r_B = 0$ and $q_B \neq 0$. 
If $a_m\neq 0$, then $C = q_BA$ does not belong to  $\RFA^0$ since $q_Ba_m \neq 0$. 
Moreover, we obtain the following contradiction:
$$
0 = B\odot C = -a_m^2q_B^2 + 2a_mq_B^2A \neq 0.
$$
Now, if $a_m = 0$, then $C = q_B$ does not belong to $\RFA^0$ and 
$$
B\odot C = q_B^2A \neq 0.
$$
Again, we obtain a contradiction with our hypothesis. Therefore, we conclude that $r_B = q_B = 0$, that is, $B = 0$.
\end{proof}

%In \cite{laiate2023properties}, the author shows that $\Psi$-cross product of two integral functions is also an integral function.  In view of this last result and Proposition \ref{prop:inequality}, we can establish the following theorem which states that the integral of the square of integral function is always greater or equal than zero. 

In view of Proposition \ref{prop:inequality}, we can establish the following theorem which states that the integral of the square of integral function is always greater or equal than zero. 

\begin{theorem}\label{thm:integral_square}
Let $A \in \RF$ be an asymmetric fuzzy number with $[A]_1 = \{a_m\}$. 
If $f:[a,b]\to \RFA$ is an integrable function such that $f(t) = r(t) + q(t)A$ for all $t \in [a,b]$, then
\begin{equation}
0 \leq \int_a^b f(t)^2 \textnormal{d}t.    
\end{equation}

In addition, $\int_a^b f(t)^2 \textnormal{d}t = 0$ if, and only if, $r(t) + q(t)a_m = 0$ almost everywhere in $[a,b]$. %Here, $f(t)^2=f(t)\odot f(t)$.
\end{theorem}
\begin{proof}
Let $r,q:[a,b]\to \RR$ such that $f(t) = r(t) + q(t)A$ for all $t \in [a,b]$. 
Note that $r$ and $q$ are also integral in $[a,b]$, 
$$
f(t)^2 = (r(t)^2 - a_m^2q(t)^2) + (2r(t)q(t) + 2a_mq(t)^2)A
$$
and 
$$
\int_a^b f(t)^2 dt = \left(\int_a^b (r(t)^2 - a_m^2q(t)^2)\textnormal{d}t \right) + \left( \int_a^b (2r(t)q(t) + 2a_mq(t)^2)\textnormal{d}t \right)A.
$$

From Corollary \ref{cor:rel_order}, $0 \leq \int_a^b f(t)^2 dt$ if, and only if, 
\begin{eqnarray*}
0 & \leq & \left(\int_a^b (r(t)^2 - a_m^2q(t)^2)\textnormal{d}t \right) +  \left( \int_a^b (2r(t)q(t) + 2a_mq(t)^2)\textnormal{d}t \right)a_m \\
& = & \int_a^b (r(t) + a_mq(t))^2\textnormal{d}t 
\end{eqnarray*}
Note that the above inequality holds true because  $(r(t) + a_mq(t))^2 \geq 0$ for all 
$t \in [a,b].$ This implies that $0 \leq \int_a^b f(t)^2 dt$. 
Moreover, we have 
$$
\int_a^b (r(t) + a_mq(t))^2\textnormal{d}t = 0 \quad \Leftrightarrow \quad r(t) + a_mq(t) = 0 
\mbox{ almost everywhere.}
$$ 
 Hence,
 $0 = \int_a^b f(t)^2 dt$ if, and only if, $f(t) \in \RFA^0$ almost everywhere. 

\end{proof}

%%%%%%%%%%%%%%%%%%%%%%
%The following lemma states that the inequality is preserved in the limit.

%\begin{lemma}
%Let $H\subseteq \RR$ and $g:H\longrightarrow \FF$ be such that $g(h)\geq 0$ for all $h\in H$. If %$\lim_{h\in H}g(h)$ exists in $\RFA$, then $\textnormal{lim}_{h\in H}g(h)\geq 0$.
%\end{lemma}
%\begin{proof}
%Consider $a_m = 0.5(\underline{a}_1 + \overline{a}_1)$, $B = (s + pA) =  \textnormal{lim}_{h\in %H}g(h) \in \RFA$, and $g(h) = r(h) + q(h)A$ for all $h\in H$. 
%Let $C_1,C_2,C_3 \subseteq H$ be the subsets of elements of $H$ such that $g(h)$ satisfies %respectively the conditions (I), (II), and (III) in Equation \ref{eq:order}, that is, 
%\begin{eqnarray*}
%C_1 & = & \{ h \in H \mid 0 < r(h) + q(h)a_m\}, \\
%C_2 & = & \{ h \in H \mid 0 = r(h) + q(h)a_m \mbox{  and  } 0 < |q(h)|\}, \\
%C_3 & = & \{ h \in H \mid 0 = r(h) + q(h)a_m \mbox{  and  } 0 = q(h)\}.
%\end{eqnarray*}
%Note that $C_1 \cup C_2 \cup C_3 = H$. 
%For $i \in \{1,2,3\}$, if $C_i \neq \emptyset$, then for all $h \in C_i$, we have that 
%$0 \leq \liminf_{h \in C_i} r(h) + q(h)a_m$. This implies that 
%$\liminf_{h \in H} r(h) + q(h)a_m \leq \lim_{h \in H} r(h) + q(h)a_m = s + pa_m$. 
%
%If $s + pa_m > 0$, then $0 \leq B$. Now, if $s+pa_m = 0$, then $0 \leq B$ since one of the %conditions (II) and (III) in \eqref{eq:order} is satisfied.
%\end{proof}
%%%%%%%%%%%%%%%%%%%%%%%%%%%%%%%%%%%%%%%5555

\section{Preliminary Results for Fuzzy Calculus of Variations}\label{Preliminary1}
\noindent

As in the classical calculus of variations, we will provide the basis for studying necessary and
sufficient conditions for general fuzzy variational problems based on the Banach space 
$\left(\RFA, \|\cdot\|\right).$
%$\left(\RFA,\oplus,\odot, \|\cdot\|\right)$. % has nonzero elements without reciprocals, that is, there exists $B\neq 0$ such that there exists no $C$ such that $B\odot C =1$ \cite{laiate2021cross}, 
%divisors and is not totally ordered (see Definition~\ref{ord}), 
%it extension for the fuzzy case requires a more detailed analysis than that for the classical case.
%Here,
%$$\Vert B\Vert_{\Psi_A}=\Vert r_B(t)+q_B(t)A\Vert_{\Psi_A}=\vert q_B(t)\vert+\vert r_B(t)+q_B(t)a_m %\vert,$$
%where $a_m$ is defined in \eqref{PontoM}.
The following lemma shows that we can interchange integration and differentiation while working with
 fuzzy functions.

\begin{lemma}\label{inter}
Given $Y\subseteq\RFA$ and $X = \left[a,b\right]\times\left[c,d\right]\subseteq\RR^2$, $a,b,c,d\in\RR$, with $a<b$ and $c<d$. Let $g:X\longrightarrow Y$ be the function defined by $g(t,\epsilon)=r(t,\epsilon)+q(t,\epsilon)A$, $\forall (t,\epsilon) \in X$, such that $r,q\in\textnormal{C}^1\left(X,\RR\right)$ (i.e., $r$ and $q$ have continuous partial derivatives in $X$). For all $\epsilon \in\left[c,d\right]$, we have 
\begin{equation}\label{dint}
	\dfrac{\textnormal{d}}{\textnormal{d} \epsilon}\int_{a}^b g(t,\epsilon)\textnormal{d}t=\int_{a}^b\dfrac{\partial}{\partial\epsilon}g(t,\epsilon)\textnormal{d}t.
\end{equation}
\end{lemma}	
\begin{proof}
It follows from the results of the classical calculus theory for the functions $r$ and $q$ together  with Propositions \ref{prop:derivative} and \ref{prop:integral}. 
\end{proof}

The notion of local minimum or maximum of a function $g\in\RFA$ is analogous to the classical case, taking into account the order relation defined in Equation~\eqref{eq:order}. More precisely, a point $t^* \in I \subset \RR$ is  
a local minimizer (or maximizer) of a function $f:I\to \RFA$ if $f(t^*) \leq f(z)$ (or $f(z) \leq f(t^*)$) for all 
$z$ in a neighborhood of $t^*$ in $I$. In this case, $f(t^*)$ is said to be a local minimum (or maximum) of $f$. 
A natural question is whether or not it is possible to establish a necessary and sufficient conditions for a point $t^*$ to be a local minimizer or maximizer of a $\RFA$-value function in terms of its derivative when it exists. The following lemmas address this issue.

\begin{lemma}\label{lem:first_condition}
Let $I\subset \RR$ be an open interval, $A\in \RFA$ be asymmetric with $a_m = 0.5(\underline{a}_1 + \overline{a}_1)$, and $f:I \to \RFA$ be a differentiable function in $I$ such that 
$f(t) = r(t) + q(t)A$, $\forall t \in I$.
If $t^* \in I$ is a local minimizer or maximizer of $f$, then $r'(t^*) + q'(t^*)a_m = 0$. 
\end{lemma}
\begin{proof}
First of all, let us suppose w.l.g. that $t^*$ is minimizer since the proof is similar for the case where $t^*$ is a maximizer. 

From Proposition \ref{prop:derivative}, we have that $r$ and $q$ are differentiable in $I$. 
Let $(t_n)$ be an arbitrary decreasing sequence such that $t_n \to t^*$ and $f(t^*) \leq f(t_n)$ for all $n$. 
It is always possible to extract a subsequence $(t_{n_k})$ such that either (I) in \eqref{eq:order} is satisfied for all $k$ or one of the conditions (II) and (III) in \eqref{eq:order} is satisfied for all $k$. Thus, 
we have that one of the following cases occurs:
\begin{itemize}
    \item[(a)] $r(t^*) + q(t^*)a_m< r(t_{n_k}) + q(t_{n_k})a_m$ for all $k$;
    \item[(b)] $r(t^*) + q(t^*)a_m = r(t_{n_k}) + q(t_{n_k})a_m$ for all $k$;
\end{itemize}
If Item (a) happens, then we have 
\begin{equation*}
0 < \frac{(r(t_{n_k}) - r(t^*)) + a_m(q(t_{n_k}) - q(t^*)) }{t_{n_k} -t^*}      
\end{equation*}
\begin{equation*}
\Rightarrow 0 \leq \lim_{k\to \infty} \frac{(r(t_{n_k}) - r(t^*)) + a_m(q(t_{n_k}) - q(t^*)) }{t_{n_k} -t^*} = r'(t^*) + a_mq'(t^*). 
\end{equation*}
Now, if (b) is satisfied, then, using the same manipulations above, we conclude that $0= r'(t^*) + a_mq'(t^*)$. 
Therefore, in both cases the inequality $0\leq r'(t^*) + a_mq'(t^*)$ holds true. 

In the same way, we can conclude that $0\geq r'(t^*) + a_mq'(t^*)$ following the same steps above, but with the sequence $(t_n)$ being increasing. 

Therefore, from these inequalities, we obtain that $0= r'(t^*) + a_mq'(t^*)$. 
\end{proof}

\begin{lemma}\label{lem:second_condition}
Let $I\subset \RR$ be open interval, $A\in \RFA$ be asymmetric with $a_m = 0.5(\underline{a}_1 + \overline{a}_1)$, and $f:I \to \RFA$ be a function three times continuously differentiable in $I$ such that 
$f(t) = r(t) + q(t)A$ for all $t \in I$. 
\begin{itemize}
    \item[a)] If $0=r'(t^*) + q'(t^*)a_m$ and $0 < r''(t^*) + q''(t^*)a_m$, then $f(t^*)$ is a local minimum of $f$;
    \item[b)] If $0=r'(t^*) + q'(t^*)a_m$ and $0 > r''(t^*) + q''(t^*)a_m$, then $f(t^*)$ is a local maximum of $f$. 
\end{itemize}
\end{lemma}
\begin{proof}
Here, we only prove Item a) since the proof of Item b) follows similarly. 

Consider the function $g(t) = r(t) + q(t)a_m$ for all $t \in I$. Note that $g$ is three times continuously  differentiable in $I$ because $r$ and $q$ are so by Proposition \ref{prop:derivative}. 
By hypothesis, we have that $g'(t^*) = r'(t^*) + q'(t^*)a_m = 0$ and 
$g''(t^*) = r''(t^*) + q''(t^*)a_m > 0.$
Using Taylor Theorem, we have 
$$
g(t^*+h) = g(t^*) + g'(t^*)h + 0.5g''(t^*)h^2 + \epsilon(h^3) =  g(t^*) + 0.5g''(t^*)h^2 + \epsilon(h^3)
$$
where $\lim_{h\to 0} \frac{\epsilon(h^3)}{h^2} = 0$. 
From the hypothesis that $I$ is open and the property of the remainder $\epsilon$,   
there exists $\delta > 0$ such that for all $|h| < \delta$ we have $t^*+h \in I$ and $\Big|\frac{\epsilon(h^3)}{h^2}\Big| \leq \frac{g''(t^*)}{2}.$
This implies that 
$$
\frac{2}{h^2}(g(t^*+h) - g(t^*)) \geq   0.5g''(t^*) > 0 \Rightarrow g(t^*+h) > g(t^*).
$$
Thus, we conclude that $f(t^*) \leq f(t^*+h)$ for all $h \in (-\delta,\delta)$ since the condition (I) in \eqref{eq:order} is satisfied:
$$
g(t^*) = r(t^*) + q(t^*)a_m < g(t^*+h) = r(t^*+h) + q(t^*+h)a_m.
$$
\end{proof}

%\begin{lemma}
%Let $h_0\in\RR$ and  $\sigma>0$. Let also $g:[h_0-\sigma,a+\sigma]\longrightarrow\FF$ be an $\textnormal{C}^2$-function such that $g'(h_0)=0$ and $g''(h_0)\ne0$. 
%If $g''(h_0)>0$ then $g(h_0)$ is a minimum local.
%\end{lemma}
%\begin{proof}
%Para provar este lema, tentei usar a fórmula de Taylor-Lagrange mas fiquei com dúvida ao utilizar os sinais de $A$.
%\end{proof}

%Another interesting result for the purpose of this article is that the integration by parts method %is valid for the  $\Psi$-cross product. The following proposition clarifies this fact. 

%\begin{proposition}\label{IPF}\cite{simoes2023interactive}
%Let $g,h$ be functions from $[a,b]$ to $\RFA$. If $g$ and $h$ are continuously diferentiable in %$[a,b]$, then 
%\begin{equation}\label{IPP}
%\int^b_a g'(t)\odot h(t)\,\textnormal{d}t=\left[g(t)\odot h(t)\right]^b_a -\int^b_a g(t)\odot %h'(t)\,\textnormal{d}t.
%\end{equation}
%\end{proposition}
%
%\begin{definition}[Fuzzy Representation]\label{def:representation}
%Let $f:D \to \RFA$, $D \subseteq \RR$. 
%We say that a function $g:D \to \RFA$ is a fuzzy representation of $f$ if 
%$f(t) - g(t) \in \RFA^0$ almost everywhere. {\bf Additionally, if $\| f(t) - g(t) \| < \epsilon$ almost everywhere for some $\epsilon > 0$, then $g$ is said a $\epsilon$-fuzzy representation of $f$.}  
%\end{definition}

\section{The Fuzzy Lagrange and du Bois-Reymond Lemmas}\label{sec:lemmas}
\noindent

As we know, Lagrange and du Bois-Reymond lemmas play a very important role in classical variational calculus and in the theory of ordinary differential equations \cite{GB,PC,E,Hes66,Logan}. In this section, we study the generalizations of the classical Lagrange and du Bois-Reymond lemmas for the class of functions on $\RFA$. We start with an extension of Lagrange classic lemma to a general fuzzy setting.

\begin{lemma}(Lagrange)\label{lem:fundamental1}
Let $A$ be an asymmetric fuzzy number such that $[A]_1 = \{a_m\}$. 
\begin{itemize}
\item[(i)] If $f:[a,b]\longrightarrow\RFA$ is an integrable function such that  
\begin{equation}\label{eq:fundamental1a}
\int_{a}^{b} f(t)\odot\eta(t)\textnormal{d}t \in \RFA^0	
\end{equation}
for any $\eta \in C^{l}\left([a,b],\RFA\right)$ with $\eta^{(j)}(a)=\eta^{(j)}(b)=0$, $j=0,1,\dots,l$, then 
$f(t) \in \RFA^0$ almost everywhere in $[a,b].$

 \item[(ii)] If $f:[a,b]\longrightarrow\RFA$ is an integrable function such that  
\begin{equation}\label{eq:fundamental1b}
\int_{a}^{b} f(t)\odot\eta(t)\textnormal{d}t=0	
\end{equation}
for any $\eta \in C^{l}\left([a,b],\RFA\right)$ with $\eta^{(j)}(a)=\eta^{(j)}(b)=0$, $j=0,1,\dots,l$, then 
$f(t) = 0$ almost everywhere in $[a,b].$
\end{itemize}
\end{lemma}
\begin{proof}
First of all, there exist integrable functions $r,q:[a,b]\to \RR$ such that $f(t) = r(t) + q(t)A$ for all $t \in [a,b]$. Since $r$ and $q$ are integrable functions, they are continuous almost everywhere in $[a,b]$, that is, they are discontinuous in subsets with zero measure. Thus, there exist a open set $I \subset [a,b]$ such that $r$ and $q$ are continuous in $I$ and  the closure of $I$ is $[a,b]$. 
Let $\epsilon> 0$, consider the functions $\delta_k:[-\epsilon,\epsilon]\to \RR$ given by 
$$
\delta_k(x) = c_k^{-1}\left[\frac{\cos\left(\pi x\epsilon^{-1}\right)+1}{2}\right]^{(l+1)k}
$$
where $c_k = \int_{-\epsilon}^\epsilon \left[\frac{\cos\left(\pi s\epsilon^{-1}\right)+1}{2}\right]^{(l+1)k} \textnormal{d}s$. Note that $(\delta_k)$ is Dirac sequence and each $\delta_k$ is $l$ times continuous differentiable.

\begin{itemize}
\item[(i)] If there exists $t \in I$ such that $f(t) \not\in \RFA^0$, then $r(t) + q(t)a_m \neq 0$.  
Since $I$ is open, there exists $\epsilon > 0$ such that $(t-\epsilon, t+\epsilon) \subset I$ and $f$ is continuous in $(t-\epsilon, t+\epsilon)$. 
Using the functions $\delta_k$, we define the functions $\eta_k(z) = r_k(z) + q_k(z)A$, $\forall z \in [a,b]$, with 
$$
r_k(z) = \left\lbrace\begin{array}{cl}
r(z)\delta_k(z-t)    &, \mbox{ if } |z-t|< \epsilon \\
0     &, \mbox{ otherwise}  
\end{array}\right.
$$
and 
$$
q_k(z) = \left\lbrace\begin{array}{cl}
q(z)\delta_k(z-t)    &, \mbox{ if } |z-t|< \epsilon \\
0     &, \mbox{ otherwise}  
\end{array}\right.
$$
for all $z \in [a,b].$

Let $B_k = \int_{a}^{b} f(z)\odot\eta_k(z)\textnormal{d}z$ and $[B_k]_1 = \{b_k\}$ with
\begin{eqnarray*}
b_k & = & \left(\int_{a}^{b} r(z)r_k(z) - a_m^2q(z)q_k(z) \textnormal{d}z\right) +  \\ &  & \, +   
\left(\int_{a}^{b} r(z)q_k(z) + r_k(z)q(z) + 2a_mq(z)q_k(z) \textnormal{d}z\right)a_m \\
& = & \left(\int_{t-\epsilon}^{t+\epsilon} (r(z)^2 - a_m^2q(z)^2)\delta_k(z-t) \textnormal{d}z\right) +   \\ &  & \, +    
\left(\int_{t-\epsilon}^{t+\epsilon} (2r(z)q(z) + 2a_mq(z)^2)\delta_k(z-t) \textnormal{d}z\right)a_m \\
& = & \int_{t-\epsilon}^{t+\epsilon} (r(z)^2 + a_mq(z))^2\delta_k(z-t) \textnormal{d}z.
\end{eqnarray*}
Since $(\delta_k)$ is a Dirac sequence, we have that $b_k \to (r(t)^2 + a_mq(t))^2 > 0$ when $k \to \infty$. Thus, for $k$ sufficient large we have that $b_k > 0$, which implies that $B_k \not\in \RFA^0$. This contradicts the hypothesis that $\int_{a}^{b} f(t)\odot\eta_k(t)\textnormal{d}t \in \RFA^0$. 
Therefore, we must have that $f(t) \in \RFA^0$ almost everywhere. 

\item[(ii)] 
%Suppose there exists $t \in I$ such that $f(t) \neq 0$.  
%This implies that $r(t) \neq 0$ or $q(t) \neq 0$
For every $t \in I$, since $I$ is open, there exists $\epsilon > 0$ such that $(t-\epsilon, t+\epsilon) \subset I$ and $f$ is continuous in $(t-\epsilon, t+\epsilon)$. 
Using the functions $\delta_k$, we define the functions $\tilde{\eta}_k(z) = \tilde{r}_k(z)$, $\forall z \in [a,b]$, with 
$$
\tilde{r}_k(z) = \left\lbrace\begin{array}{cl}
\delta_k(z-t)    &, \mbox{ if } |z-t|< \epsilon \\
0     &, \mbox{ otherwise}  
\end{array}\right.$$
for all $z \in [a,b].$
From our hypothesis, we have
\begin{eqnarray*}
0 & = & \left(\int_{a}^{b} r(z)\tilde{r}_k(z) \textnormal{d}z\right) +     
\left(\int_{a}^{b} \tilde{r}_k(z)q(z)\textnormal{d}z\right)A \\
& = & \left(\int_{t-\epsilon}^{t+\epsilon} r(z)\delta_k(z-t) \textnormal{d}z\right) +     
\left(\int_{t-\epsilon}^{t+\epsilon} q(z)\delta_k(z-t)\textnormal{d}z\right)A.
\end{eqnarray*}
Using the above equality and the fact that $(\delta_k)$ is Dirac sequence, we obtain 
$$\Rightarrow \left\lbrace \begin{array}{l}
0 = \int_{t-\epsilon}^{t+\epsilon} r(z)\delta_k(z-t) \textnormal{d}z \xrightarrow[k\to \infty]{} r(t)     \\ 
0 = \int_{t-\epsilon}^{t+\epsilon} q(z)\delta_k(z-t) \textnormal{d}z \xrightarrow[k\to \infty]{} q(t)
\end{array}\right..
$$
%For $k$ sufficient large, this produces a contradiction with the hypothesis that $\int_{a}^{b} \tilde{\eta}_k(z)\odot f(z) \textnormal{d}z = 0$, since $r(t) \neq 0$ or $q(t) \neq 0$. 
Therefore, we conclude that $f(t) = 0$ almost everywhere in $[a,b].$
\end{itemize}
\end{proof}

\begin{corollary}(du Bois-Reymond's lemma I) \label{cor:fundamental2}
Let $A$ be an asymmetric fuzzy number such that $[A]_1 = \{a_m\}$. 
\begin{itemize}
\item[(i)] If $f:[a,b]\longrightarrow\RFA$ is a differentiable function almost everywhere in $[a,b]$  such that  
\begin{equation}\label{eq:fundamental2a}
\int_{a}^{b} f(t)\odot\eta'(t)\textnormal{d}t \in \RFA^0	
\end{equation}
for any $\eta \in C^1\left([a,b],\RFA\right)$ with $\eta(a)=\eta(b)=0$, then 
$f(t) - u \in \RFA^0$ almost everywhere in $[a,b]$ for some $u \in \RFA$.

\item[(ii)] If $f:[a,b]\longrightarrow\RFA$ is a differentiable function almost everywhere in $[a,b]$  such that  
\begin{equation}\label{eq:fundamental2b}
\int_{a}^{b} f(t)\odot\eta'(t)\textnormal{d}t = 0	
\end{equation}
for any $\eta \in C^1\left([a,b],\RFA\right)$ with $\eta(a)=\eta(b)=0$, then 
$f(t) = u$ almost everywhere for some $u \in \RFA$.

%\item[(iii)] If $f:[a,b]\longrightarrow\RFA$ is an integrable function such that  
%\begin{equation}\label{eq:fundamental1b}
%\int_{a}^{b} f(t)\odot\eta'(t)\textnormal{d}t=0	
%\end{equation}
%for any $\eta \in \textnormal{Lip}\left([a,b],\FF\right)$ with $\eta(a)=\eta(b)=0$, then 
%$f(t) - u \in \RFA^0$ almost everywhere in $[a,b]$ for some $u \in \RFA$.
\end{itemize}
\end{corollary}
\begin{proof}
Let $r,q:[a,b]\to \RR$ be such that $f(t) = r(t) + q(t)A$ for all $t \in [a,b]$.
 Using the facts that $f$ and $\eta$ are differentiable almost everywhere and therefore continuous almost everywhere, we can apply the integration by parts to obtain:
 $$
 \int_{a}^{b} f(t)\odot\eta'(t)\textnormal{d}t = f(t)\odot\eta(t) \Big|_a^b \ominus \int_{a}^{b} f'(t)\odot\eta(t)\textnormal{d}t = - \int_{a}^{b} f'(t)\odot\eta(t)\textnormal{d}t.
 $$
Note that $f(t)\odot\eta(t) \Big|_a^b = 0$ because $\eta(a) = \eta(b) = 0.$
\begin{itemize}
\item[(i)] Using the hypotheses of the corollary and the equality above, we have 
 $$
 \int_{a}^{b} f'(t)\odot\eta(t)\textnormal{d}t \in \RFA^0
 $$
for all $\eta$ satisfying the conditions of the enunciate. From Item (i) of Lemma \ref{lem:fundamental1}, 
we have that $f'(t) \in \RFA^0$ almost everywhere. This implies that 
$r'(t) + q'(t)a_m = 0$ almost everywhere and therefore 
$r(t) + q(t)a_m = u_m$ almost everywhere, for some $u_m \in \RR$. 
For every $u \in \RFA$ such that $u_m = 0.5(\underline{u}_1 + \overline{u}_1)$, the claim of Item (i) holds true. 

\item[(ii)] Again, from our hypotheses, we have 
$$
 \int_{a}^{b} f'(t)\odot\eta(t)\textnormal{d}t = 0
 $$
for all $\eta \in C^1([a,b],\RFA)$ with $\eta(a)= \eta(b) = 0$. 
From Item (ii) of Lemma \ref{lem:fundamental1}, we conclude that $f'(t)  = 0$, that is, $r(t) =  q'(t)=0$ almost everywhere.
This implies that $f(t) = u$ almost everywhere, for some $u \in \RFA$. 
\end{itemize}
\end{proof}

\begin{lemma}(du Bois-Reymond's lemma II)\label{lem:fundamental3} 
Let $A$ be an asymmetric fuzzy number such that $[A]_1 = \{a_m\}$. 
If $f:[a,b]\to \RFA$ is a integrable function such that 
\begin{equation}\label{eq:aux_p}
\int_{a}^{b}f(t) \odot \eta'(t)\,\textnormal{d}t=0	
\end{equation}
for any $\eta \in C^1\left([a,b],\RFA\right)$ with $\eta(a) = \eta(b)=0$, 
then $f(t) \ominus u \in \RFA^0$ almost everywhere in $[a,b]$, where  $u \in \RFA.$
\end{lemma}
\begin{proof}
Let $u = v + wA = (b-a)^{-1}\int^b_a f(t)\,\textnormal{d}t$ and let $r,q:[a,b]\to \RR$ be 
such that $f(t) = r(t) + q(t)A$ for all $t \in [a,b].$ 
Since $f$ is integrable, the functions $r$ and $q$ are also integrable. This implies that $r$ and $q$ are continuous almost everywhere and the functions $\tilde{r}, \tilde{q}:[a,b]\to \RR$ given below are continuous in $[a,b]$:
$$
\tilde{r}(t) = \left\lbrace\begin{array}{cl}
\lim_{s\to t^-} r(s)   &, \mbox{ if } t \in (a,b] \\
\lim_{s\to a^+} r(s)   &, \mbox{ if } t=a
\end{array}\right.
$$
and
$$
\tilde{q}(t) = \left\lbrace\begin{array}{cl}
\lim_{s\to t^-} q(s)   &, \mbox{ if } t \in (a,b] \\
\lim_{s\to a^+} q(s)   &, \mbox{ if } t=a
\end{array}\right..
$$ 
Consider the function $\tilde{f}(t) = \tilde{r}(t) + \tilde{q}(t)A$ and 
$$
\eta(t) = \int_a^t (\tilde{f}(s) \ominus u)ds
$$
for all $t \in [a,b].$
Note that  $\eta$ is continuously differentiable in $[a,b]$ with 
$\eta(a) = \eta(b) = 0$ and $\eta'(t) = \tilde{f}(s) \ominus u = f(s) \ominus u$ 
almost everywhere because $r(t) = \tilde{r}(t)$ and $q(t) = \tilde{q}(t)$ almost everywhere. 
From Theorem \ref{thm:integral_square}, we have 
\begin{eqnarray*}
0 & \leq & \int_{a}^{b}(f(t) \ominus u) \odot (f(t) \ominus u)\,\textnormal{d}t \\
 & = & \int_{a}^{b} (f(t) \ominus u) \odot \eta'(t) \,\textnormal{d}t \\
 & = & \int_{a}^{b} f(t) \odot \eta'(t) \,\textnormal{d}t \ominus u \odot \int_{a}^{b} \eta'(t) \,\textnormal{d}t = 0.  
\end{eqnarray*}
The last equality follows from the hypothesis that $\int_{a}^{b} f(t) \odot \eta'(t) \,\textnormal{d}t = 0$ and $u \odot \int_{a}^{b} \eta'(t) = u\odot \eta(t) \Big|_a^b = 0$ because 
$\eta(a)=\eta(b) =0$. 
Again, from the second part of Theorem \ref{thm:integral_square}, we conclude that 
$f(t) \ominus u \in \RFA^0$ almost everywhere. 
\end{proof}

In what follows we present a generalization of du Bois--Reymond lemma for $\RFA$-valued functions. 
%Recall the set of real numbers are embedded in $\RFA$ and, therefore, every real-valued function 

\begin{theorem}(Fuzzy du Bois--Reymond lemma)\label{DR}
Let $A$ be an asymmetric fuzzy number such that $[A]_1 = \{a_m\}$ and let $f,g:[a,b]\longrightarrow\RFA$ be two functions defined respectively by $f(t)=r_f(t)+q_f(t)A$ and $g(t)=r_g(t)+q_g(t)A$. The following statements hold:  
\begin{itemize}
\item[(i)] If $f,g$ are integrable functions and
\begin{equation}\label{ldM}
\int_{a}^{b}\left[f(t)\odot\eta(t) \oplus g(t) \odot \eta'(t)\right]\,\textnormal{d}t=0	
\end{equation}
for any $\eta \in C^1\left([a,b],\RFA\right)$ with $\eta(a)=\eta(b)=0$, then there exists a continuous
function $\tilde{g}$ such that $\tilde{g}(t) \ominus g(t) \in \RFA^0$ and 
$\tilde{g}'=f$ almost everywhere.
	
\item[(ii)] If $f,g$ are continuous functions and $g$ is differentiable almost everywhere such that 
\begin{equation*}\label{ldM}
\int_{a}^{b}\left[f(t)\odot\eta(t) \oplus g(t) \odot \eta'(t)\right]\,\textnormal{d}t=0	
\end{equation*}
for all $\eta \in \textnormal{C}^1\left([a,b],\RFA\right)$ with $\eta(a) = \eta(b) = 0$, then 
$g' = f$ almost everywhere.

\end{itemize}
\end{theorem}
\begin{proof}
Define $F(t) = r_F(t) + q_F(t)A = \int_a^t f(s)ds$ for all $t \in [a,b].$  
Note that $F'(t)=f(t)$ almost everywhere. Using the hypotheses of theorem, the product rule, and integration by parts, we obtain 
\begin{align*}
0&=\int_{a}^{b}\left[f(t)\odot\eta(t) \oplus g(t) \odot \eta'(t)\right]\textnormal{d}t\\
&=\int_{a}^{b}\dfrac{\textnormal{d}}{\textnormal{d}t}\left[F(t)\odot\eta(t)\right]\textnormal{d}t \oplus \int_{a}^{b}\left[g(t)-F(t)\right]\odot\eta'(t)\textnormal{d}t \\
&=\int_{a}^{b}\left[g(t) \ominus F(t)\right]\odot\eta'(t)\textnormal{d}t.	
\end{align*}

\begin{itemize}
\item[(i)] From Lemma \ref{lem:fundamental3}, there exists $u =v+wA\in \RFA$ such that 
$g(t)\ominus F(t) \ominus u \in \RFA^0$ almost everywhere, that is, 
$$
(r_g(t) - r_F(t) - v) + (q_g(t) - q_F(t) - w)a_m = 0.
$$
Let $\tilde{g}(t) = F(t) \oplus u$ for all $t \in [a,b]$. Note that $\tilde{g}$ is continuous in $[a,b]$  and differentiable with $\tilde{g}' = f$ and $g(t) \ominus \tilde{g}(t) \in \RFA^0$ almost everywhere.

\item[(ii)] From Item (ii) of Corollary \ref{cor:fundamental2}, there exists $u \in \RFA$ such that 
$g(t) \ominus F(t) = u$ almost everywhere. 
This implies that $g(t) = F(t) \oplus u$ and $g'(t) = f(t)$ almost everywhere.
\end{itemize}
\end{proof}

\section{Final Remarks}

To the best of our knowledge, this is the first article that uses the notion of total order in the study of fuzzy variational  calculus. This concept is very important since the fundamental problems in this area deal with minimization/maximization of a functional. Another key property that is often used in the classical case is that the square a real number is non negative. % and that the integral of the square of a function is zero implies that this function is zero almost everywhere. 
Using the proposed order relation and the notion of multiplication $\odot$ ($\Psi$-cross product), we show that $B\odot B$=$B^2\geq 0$ for all fuzzy number $B$. This property is very useful to check several important results such as lemmas of Lagrange and du Bois-Reymond (see Section \ref{sec:lemmas}). 

Our study focuses on the space $\RFA$ that is composed by all fuzzy numbers that are linearly correlated with an asymmetric fuzzy number $A$ (i.e., $\RFA = \{r+qA\mid (r,q) \in \RR^2 \}$. This space has two important mathematical characteristics: 1) It is a Banach space, which is an ideal framework for developing a theory of differential and integral calculus and, at the same time, 2) it allows us to model phenomena such that uncertainties (or noise) cannot be disregarded in the task of obtaining an adequate solution for the problem at hand.   

Since each $B \in \RFA$ can be written as $B = r+qA$, in analogy to the language used in statistics (or even in stochastic calculus), we can understand $B$ as the ``composition'' of the real part ``$r$'' and the uncertain part ``$qA$''. In this context, we can interpret the elements of $\RFA$ from an epistemic point of view, which suggests a certain hierarchy of importance of the parameters $r$ and $q$ in the proposed notion of order on $\RFA$ (see Proposition \ref{Proorder}). 
Note that, between two fuzzy numbers, the greatest is that with the greatest center ($r+qa_m$), that is, the greatest middle point of the $1$-level. In the case of a tie, the modulus of the parameter $q$, which concerns the width of the noise ($qA$), decides which one is the greatest. 
Thus, since the centers of the fuzzy numbers are used to decide the greatest between two elements,  
the proposed order extends the usual order on the set of real numbers.  

We would like to present some brief comments comparing our proposal with the theory of fuzzy calculus based on the concept of generalized Hukuhara derivative or, for short, $gH$-derivative, which is the well-known and most widely used in the literature. On the one hand, in a first glance, the notion of genera\-lized Hukuhara derivative seems to be more general since it is not restricted to special classes of fuzzy numbers \cite{bede13}. 
However, its use requires the identification {\it a priori} of the switching points of the function that, roughly speaking, are the points of the domain that occur changes in the behavior of the widths of the $\alpha$-levels. In addition, it also requires the application of Stacking Theorem \cite{negoita75} to ensure its well-definition in at certain point, 
since, in general, its calculation is done in the endpoints of $\alpha$-levels. 
It is worth noting that the verification of the hypotheses of Stacking Theorem 
is a very hard task in practice. Thus, although no restrictions are imposed on the codomain, its use could be restrictive due to the comments above. Furthermore, to the best of our knowledge, the use of some notion of order relation in the study of fuzzy variational calculus under $gH$-derivative is not clear, which can make it difficult to understand the issue of minimizing/maximizing a functional. 
Beside the $gH$-derivative, there exist other notions of fuzzy derivative (see, e.g., \cite{khastan2022new})  for which one could study variational problems, but each one could lead us to different approaches to fuzzy variational calculus.  
However, there are still few approaches to fuzzy variational calculus in the literature \cite{verma2022systematic}.
%Como é bem conhecido, there exist several concept of derivative in fuzzy context (see New metric-based derivatives for fuzzy functions and some of their properties, for example), cada uma destas poderia resultar em diferentes teorias sobre fuzzy optimization.

%On the other hand, our approach to fuzzy calculus can be applied to functions taking values in  subclasses of fuzzy numbers $\RFA$. These subclasses are subsets of fuzzy numbers whose elements are autocorrelated. Among them, we could highlight those based on linear (or complete) correlation \cite{carlsson2004additions} for which the space $\RFA$ is included \cite{de2021differential}. 
%Under certain conditions, these spaces have structure of Banach spaces and, therefore, they are closed with respect to the addition, subtraction, and scalar product. This fact eliminates the need to examine the requirements mentioned above when using the $gH$-derivative. Indeed, the theory of differential and integral calculus in Banach space is already well-established. 
%Moreover, como comentado acima, aqui temos uma noção clara de ordem total e de produto cuja combinação produz a relação $B\odot B$=$B^2\geq 0$.

Finally, the next steps of our research follow in the direction of the extension of the main results of the calculus of variations to the fuzzy case, such as the formalization of optimization of fuzzy functional taking values in $\RFA$. Note that the order relation on the corresponding function space is given as a natural extension of the order relation on $\RFA$. These results and others will be presented in the other parts of this series of articles on Fuzzy Variational Calculus in Linearly Correlated Space. 

%we destacamos que com nossa abordagem, conseguimos classe de numeros fuzzy com estutura matemática bastante rica, no sentido de ser comparada com casos clássicos. Tal estrutura nos parece bastante promissor para trabalhar com o cálculo variacional para funções fuzzy, como pode ser constatado na seção 5, em que formulamos os The fuzzy lemmas of Lagrange and du Bois Reymond. 

%NOSSO CALCULO EM RF(A), JUNTO COM ORDEM E O PRODUTO FOI POSSIVEL FAZER EULER E DU BOIS ... QUE É A BASE PARA O CALCULO VARIACIONAL, DE MANEIRA TOTALMENTE SEMELHANDO AO CASO DE CALCULO VARIACIONAL PARA FUNÇÕES REAIS, EM QUE MINIMIZAR (PORTANTO, NECESSIDADE DE ORDEM) UM FUNCIONAL É A ARGUMENTAÇÃO PRINCIPAL

\section*{Acknowledgements}

The research has been partially supported by FAPESP under grant numbers 2020/09838-0 and 2022/00196-1 and by CNPq under grant numbers 314885/2021-8 and 311976/2023-9.

%\bibliographystyle{acm}
%\bibliography{references}

\end{document}